\documentclass[12pt]{amsart}
\usepackage{amssymb}
\usepackage{amsthm}
\usepackage{amsmath}
\usepackage{amsfonts}
\usepackage{graphicx}
\usepackage{setspace}
\usepackage[all]{xy}
\begin{document}
\newtheorem{thmA}{Theorem}
\renewcommand{\thethmA}{\Alph{thmA}}
\newtheorem{corA}[thmA]{Corollary}


\newtheorem{theorem}{Theorem}[section]
\newtheorem{thm}[theorem]{Theorem}
\newtheorem{lemma}[theorem]{Lemma}
\newtheorem{cor}[theorem]{Corollary}
\newtheorem{prop}[theorem]{Proposition}
\newtheorem{addendum}[theorem]{Addendum}
\newtheorem{claim}[theorem]{Claim}
\newtheorem{question}[theorem]{Question}

\newtheorem{thmspec}{\relax}

\theoremstyle{definition}
\newtheorem{definition}[theorem]{Definition}
\newtheorem{example}[theorem]{Example}

\theoremstyle{remark}
\newtheorem{remark}[theorem]{Remark}
\newtheorem{remarks}[theorem]{Remarks}

\newenvironment{thmbis}[1]
  {\renewcommand{\thethmA}{\ref{#1}$'$}%
   \addtocounter{thmA}{-1}%
   \begin{thmA}}
  {\end{thmA}}

\numberwithin{equation}{section}

\hoffset -1cm
\voffset 1cm
\textheight 21cm \textwidth 14cm

\def\E{\Lambda}

\def\eg{A}
\def\er{R}
\def\P{\Pi}
\def\<{\langle}
\def\>{\rangle}
\def\A{\mathcal A}
\def\R{\mathcal R}
\def\Y{\mathcal Y}
\def\H{{\rm{H}}}

\def\th{\theta}
\def\Th{\Theta}

\def\C{\mathcal C}

\def\-{\overline}

\def\La{\Lambda}

\def\cech{\check}

\def\T{T}
\def\U{U}

\def\a{\alpha}
\def\b{\beta}
\def\c{\gamma}
\def\d{\delta}

\def\X{\mathcal X}

\def\inv{^{-1}}
\def\-{\underline}
\def\Z{\mathbb{Z}}
\def\N{\mathbb{N}}
\def\G{\mathcal G}
\def\S{\Sigma}
\def\e{\varepsilon}
\def\ln{\<\!\<}
\def\rn{\>\!\>}
\def\AR{\langle A\mid R\rangle}

\newcommand{\id}{\rm{id}}

\def\C{\mathcal{C}}

\catcode`\@=11
\def\serieslogo@{\relax}
\def\@setcopyright{\relax}
\catcode`\@=12

\catcode`\@=11
\def\serieslogo@{\relax}
\def\@setcopyright{\relax}
\catcode`\@=12

\hoffset -1cm
\voffset 1cm
\textheight 21cm \textwidth 14cm

\author[Owen Cotton-Barratt]{Owen Cotton-Barratt}
\address{Mathematical Institute, 24-29 St Giles', Oxford OX1 3LB, UK}
\email{cotton-b@maths.ox.ac.uk}

\title{Detecting ends of residually finite groups in profinite completions}

\date{March 2012}

\begin{abstract}

Let $\C$ be a variety of finite groups. We use profinite Bass--Serre theory to show that if $u:H\hookrightarrow G$ is a map of finitely generated residually $\C$ groups such that the induced map $\hat{u}:\hat{H}\rightarrow\hat{G}$ is a surjection of the pro-$\C$ completions, and $G$ has more than one end, then $H$ has the same number of ends as $G$. However if $G$ has one end the number of ends of $H$ may be larger; we observe cases where this occurs for $\C$ the class of finite $p$-groups.


We produce a monomorphism of groups $u:H\hookrightarrow G$ such that: either $G$ is hyperbolic but not residually finite; or $\hat{u}:\hat{H}\rightarrow\hat{G}$ is an isomorphism of profinite completions but $H$ has property (T) (and hence (FA)), but $G$ has neither. Either possibility would give new examples of pathological finitely generated groups.

\end{abstract}

\maketitle

\section{Introduction}
A property of finitely generated, residually finite groups is said to be profinite if it is preserved among groups with the same profinite completion. Understanding which properties are profinite amounts to understanding what the profinite completion of a residually finite group tells you about the group, and there has been a lot of recent work in the area, \emph{e.g.} \cite{Lackenby}, \cite{cotton2009conjugacy}, \cite{menny2010profinite}.


The number of ends of a finitely generated group is always 0, 1, 2 or $\infty$, and this provides an important division in classifying finitely generated groups. While 0-ended obviously coincides with finite and 2-ended is easily seen to be virtually cyclic, a theorem of Stallings precisely describes those groups with infinitely many ends \cite{Stallings}. The theorem is proved by invoking the theory of group actions on trees (see for example \cite{Serre} for an introduction to this subject).

Analogous to the discrete case and the theory of groups acting on trees, there is a theory of profinite groups acting (continuously) on profinite trees. The groundwork of this theory has been developed, principally by Ribes and Zalesskii \cite{RZ}, \cite{profinitetrees}.

We use this profinite Bass--Serre theory to prove (Theorem \ref{t: profinite ends}, below) that when an injection $u: H \hookrightarrow G$ of finitely generated, residually finite groups induces a surjection of profinite completions, then a splitting of $G$ over a finite subgroup can be used to construct such a splitting of $H$, and hence if $G$ has more than one end then $H$ must also. Thus having infinitely many ends is a \emph{down-weak profinite property}. This argument is general in that it extends to the case where the groups are residually $\C$ and the completion taken is a pro-$\C$ completion, for $\C$ any extension-closed class of finite groups.

We are unable to establish whether the number of ends of a group is fully a profinite property. However, we show via an example from the theory of parafree groups that having infinitely many ends is not a pro-$p$ property, or even an \emph{up-weak} one; hence any argument that it is a profinite property could not be made generic. That having more than one end passes via pro-$\C$ equivalence from an ambient group to a subgroup but not \emph{vice-versa} may be thought of as loosely analogous to the fact that a subgroup of a free group is free (whereas a supergroup of a free group need not be free).

\begin{thmA}\label{t: profinite ends}
If $\C$ is a variety of finite groups (\emph{i.e.} a class closed under taking subgroups, quotients, and extensions), then having more than one end is a down-weak pro-$\C$ property. However for any prime $p$, having more than one end is not an up-weak pro-$p$ property, even among finitely presented groups.
\end{thmA}

The Rips construction is an algorithmic method which takes a finitely presented group $Q$ as input and outputs a short exact sequence $1 \rightarrow N \rightarrow G \rightarrow Q \rightarrow 1$, where $N$ is finitely generated and $G$ is a hyperbolic small-cancellation group. More recent variations on the construction offer further control over $N$ or $G$. In particular, Wise produced a version in which $G$ is residually finite \cite{Wise1} (and Wilton and I showed that it is in fact conjugacy separable \cite{cotton2009conjugacy}). On the other hand, Belegradek and Osin \cite{BelOsin} exhibit a version of the Rips construction where $N$ may be taken to be a quotient of an arbitrary non-elementary hyperbolic group. We investigate whether these two versions may be reconciled. It seems that the methods used to prove residual finiteness have no hope of being applied to a variation which gives such broad control over $N$. This raises the question of whether there is a genuine obstruction to $G$ being residually finite.

We apply the Belegradek--Osin version, taking $N$ to have property (T), to a suitable pathological group $Q$, so as to exhibit a hyperbolic group $G$. Either $G$ is not residually finite, which would be interesting in its own right, or the inclusion $N \hookrightarrow G$ induces an isomorphism of profinite completions, but $N$ has properties (T) and (FA) whereas $G$ does not. Hence:

\begin{thmA}\label{t: unknown pathologies}
Either not every Gromov hyperbolic group is residually finite, or neither of the following are up-weak profinite properties: property (FA); property (T).
\end{thmA}

This work forms part of the author's doctoral thesis. He would like to thank his supervisor Martin Bridson for guidance and helpful conversations, and the EPSRC for providing support.

\section{Preliminaries}
\begin{definition}
Two finitely generated residually finite groups are said to be \emph{profinitely equivalent} if their profinite completions are isomorphic (\emph{a priori} as topological groups, but by a result of Nikolov and Segal \cite{Nikolov-Segal} this is the same as being abstractly isomorphic). A property of finitely generated residually finite groups is called a \emph{profinite property} if it can be detected in the profinite completion, \emph{i.e.} if a group $G$ has the property then so does every group profinitely equivalent to $G$.
\end{definition}

Observe that profinite equivalence is indeed an equivalence relation. Some authors talk about the equivalence class of $G$ in this sense (possibly restricted to groups in a given class $\mathcal{C}$). This is known as the \emph{genus} of $G$ (in $\mathcal{C}$).

In the cases we consider in this paper we will generally know not just that there is an abstract isomorphism between the profinite completions of two residually finite groups, $G$ and $H$, but that there is an explicit homomorphism $u:H \rightarrow G$ such that the induced map $\hat{u}:\hat{H}\rightarrow\hat{G}$ is an isomorphism of the profinite completions. Indeed in this case, since $H$ is residually finite it injects into its profinite completion, and hence $u$ is also an injection. It is not \emph{a priori} obvious that there is room for difference here: Grothendieck asked whether $u$ itself must be an isomorphism if $H$ and $G$ are finitely presented, but this question was answered in the negative in \cite{BridGrun}. We will also be interested in cases where $u$ is an injection and $\hat{u}$ is surjective (but not necessarily injective). We make two more definitions. 

\begin{definition}
A property $P$ of discrete groups is said to be an \emph{up-weak profinite property} if whenever there exists a map of residually finite groups $u:G \hookrightarrow H$ such that the induced map $\hat{u}:\hat{G}\rightarrow\hat{H}$ is an isomorphism of the profinite completions and $G$ has $P$, then $H$ must also have $P$.
\end{definition}

\begin{definition}
A property $P$ of discrete groups is said to be a \emph{down-weak profinite property} if whenever there exists a map of residually finite groups $u:H \hookrightarrow G$ such that the induced map $\hat{u}:\hat{H}\rightarrow\hat{G}$ is an isomorphism of the profinite completions and $G$ has $P$, then $H$ must also have $P$.
\end{definition}

The use of the word \emph{weak} reflects the fact that any profinite property automatically satisfies these weaker conditions. The further labels refer to whether the property is passed from the subgroup \emph{up} to the ambient group, or from the ambient group \emph{down} to the subgroup.

Any of these terms may also be used in the context where we care not about arbitrary finite quotients, but finite quotients in a certain class $\C$. In this case we replace the word `profinite' with `pro-$\C$' wherever it appears, so a pro-$\C$ equivalence means an isomorphism between pro-$\C$ completions, \emph{etc.} We will chiefly concern ourselves with the cases where $\C$ is the class of all finite groups or the class of finite $p$-groups (pro-$\C$ has special names in these cases: profinite and pro-$p$ respectively).

\subsection{Ends of groups and Stallings' Theorem}
The number of ends of a finitely generated group is an important geometric property, and the following theorem of Stallings \cite{Stallings} translates this into algebraic structure.

\begin{theorem}\label{t: Stallings}
A finitely generated group $G$ has more than one end if and only if it splits nontrivially as an amalgamated free product or HNN extension over a finite subgroup.
\end{theorem}
\subsection{Profinite Bass--Serre Theory}
We recall some basic definitions and facts from the theory of profinite groups acting on profinite trees, as developed by Ribes and Zalesskii. Details of this background material may be found in \cite{profinitetrees}, \cite{RZ}, or \cite{zalesskii_fundamental_1989}.

Recall that a \emph{profinite graph} $\Gamma$ is an inverse limit of finite graphs. It is \emph{connected} if each of its finite quotients is connected, and its \emph{profinite fundamental group} $\Pi_1(\Gamma)$ is the group of deck transformations of its profinite universal cover $\tilde{\Gamma}$ (satisfying an appropriate universal property among connected Galois covers of $\Gamma$). We say that $\Gamma$ is a (profinitely simply connected) \emph{profinite tree} if $\Pi_1(\Gamma) = 1$.

If a profinite group $G$ acts continuously on a profinite tree $T$ with \emph{finite} quotient $\Gamma$ then there is (analogously to the abstract case) a profinite graph of groups $\mathcal{G}$ over $\Gamma$ where the vertex and edge groups are isomorphic to the stabilizers of the corresponding vertices and edges in $T$; in this case $\Pi_1(\mathcal{G})$ is isomorphic to $G$. Similarly, associated to a finite\footnote{For technical reasons, this is not so straightforward when the underlying graph may not be finite.} graph of profinite groups is a profinite tree on which the fundamental group acts. As in the discrete case, each closed subgroup has a minimal invariant subtree (see Lemma 2.2 of \cite{RZ}).

The following proposition is a variation on a statement in \cite{RZ}, concerning the case where $\C$ is the class of all finite groups. This more general form must be known to the experts, but I have not found it in the literature.

\begin{prop}\label{p: profinite splitting}
Let $\C$ be an extension-closed class of finite groups. Let $\G$ be a finite graph of (discrete) groups over an underlying graph $\Gamma$, with finitely generated vertex groups, such that $G=\pi_{1}(\G)$ is residually-$\C$. Then the pro-$\C$ completion $\hat{G}$ of $G$ is isomorphic to the pro-$\C$ fundamental group $\Pi_{1}^{\C}(\hat{\G}_{\C})$ of the finite graph of pro-$\C$ groups $\hat{\G}_{\C}$ obtained from $\G$ by taking the completion of each edge and vertex group with respect to the topology induced by the pro-$\C$ topology of $G$.

There is a natural map from the standard tree $S$ associated to $\G$ to the standard pro-$\C$ tree $\hat{S}$ associated to $\hat{\G}_{\C}$. If the edge groups of $\G$ are separable in the vertex groups with respect to the topology induced by the pro-$\C$ topology on $G$, this map is an injection.
\end{prop}

\begin{proof}
The first part of the proposition simply rests on the observation that the finite quotients of $\Pi_{1}^{\C}(\hat{\G}_{\C})$ are precisely those of $G$. For in each case a finite quotient is generated by the images of the vertex groups and the fundamental group of the underlying graph. By construction the possible images (in such finite quotients) of the vertex groups from $\G$ and $\hat{\G}_{\C}$ coincide; likewise the finite quotients of the fundamental groups.

The standard tree $S$ is built as a set of right cosets of the edge and vertex groups of $\G$ in $G$. Similarly $\hat{S}$ as a set is the disjoint union of right cosets of the edge and vertex groups (of $\hat{\G}_{\C}$) in $\hat{G}$. The obvious map sending a coset $G_{\gamma}g$ to the coset $\hat{G}_{\gamma}g$ respects incidence maps and so is a map of graphs $S \rightarrow \hat{S}$. The natural requirement for this to be an injection is that the vertex and edge groups are separable in $G$ with respect to the pro-$\C$ topology on $G$, or in other words that the pro-$\C$ topology on $\G$ is efficient. 

Our assumption that the edge groups of $\G$ are separable in the vertex groups is \emph{a priori} weaker than the efficiency of $\G$. When the edge groups are separable in the vertex groups, however, the edges adjoining any vertex $v$ in $S$ have disjoint images in $\hat{S}$. So the map from $S$ to $\hat{S}$ is \emph{locally} injective. If it were not globally injective there would be some finite loops in the image. But $\hat{S}$ is a pro-$\C$ tree, and in particular has no finite loops; hence separability of the edge groups in the vertex groups is sufficient to give an injection of standard trees.
\end{proof}

\section{Ends of groups}\label{s: ends profinite property}
Let $\mathcal{C}$ be an extension-closed class of finite groups. We will show that having more than one end is a down-weak pro-$\mathcal{C}$ property. We refer the reader back to Proposition \ref{p: profinite splitting}, which lets us pass a splitting of a pro-$\mathcal{C}$ group to a splitting of its pro-$\mathcal{C}$ completion.

\begin{lemma}\label{l: profinite splitting}
Let $u:H \hookrightarrow G$ be a map of residually $\mathcal{C}$ groups such that the induced map $\hat{u}:\hat{H}\rightarrow\hat{G}$ is a surjection of the pro-$\mathcal{C}$ completions and $G$ splits nontrivially as the fundamental group of a finite graph of groups $\mathcal{G}$ over the underlying graph $\Gamma$ with edge groups separable in the vertex groups they are incident at (with the topology induced by the pro-$\mathcal{C}$ topology on $G$). Then $H$ also splits as a finite graph of groups.
\end{lemma}

\begin{proof}
Let $T$ be the Bass--Serre tree associated with the splitting of $G$ given by $\mathcal{G}$. As $G$ is residually $\mathcal{C}$, by Proposition \ref{p: profinite splitting} the pro-$\mathcal{C}$ completion $\hat{G_{\mathcal{C}}}$ of $G$ can be recovered as the pro-$\mathcal{C}$ fundamental group of the completion $\hat{\mathcal{G}_{\mathcal{C}}}$, so the quotient $\hat{G}\backslash \hat{T}$ of the pro-$\mathcal{C}$ tree $\hat{T}$ is isomorphic to $\Gamma$. Moreover since the edge groups are separable in the vertex groups, $T$ embeds into the pro-$\mathcal{C}$ tree $\hat{T}$. As $H$ is a subgroup of $G$ it also acts on $T$, and we consider the quotient $H \backslash T$. $H$ cannot lie in the vertex stabilizer of any vertex in $T$, since then by continuity its closure $\bar{H}$ would also lie in the vertex stabilizer, which is absurd since $\bar{H}\backslash \hat{T} = \Gamma$. Hence $H$ splits nontrivially as a graph of groups. 

There is some edge in this graph of groups such that by collapsing everything except this one edge we get a different nontrivial splitting of $H$, as the fundamental group of a graph of groups with a single edge. The existence of such an edge may be seen by an argument by contradiction.

Firstly, if the graph of groups is not a tree of groups, then by choosing an edge in a loop and collapsing the connected components at either end, we get a decomposition of $H$ as an HNN extension, which is necessarily non-trivial. If we have a tree of groups, begin by choosing an edge. If collapsing the connected components at either end we reduce to a trivial splitting, then $H$ is contained in one of the vertex groups of this induced splitting. Now we choose an edge in the corresponding component and repeat. There are \emph{a priori} two possibilities: this process will terminate with a non-trivial splitting for $H$ over a single edge; or it does not terminate, and (by collapsing all except the sequence of chosen edges) we get $H$ as the ascending union of a ray of groups. But in this latter case $H$ is inside the ascending union of the edge groups; each edge group is a subgroup of one of the finitely many edge groups of $\G$, so $H$ lives inside one of the edge groups, and hence one of the vertex groups, of $\G$; a contradiction.
\end{proof}

\begin{prop}\label{p: ends profinite}
Whenever $u:H \hookrightarrow G$ is a map of residually $\mathcal{C}$ groups such that the induced map $\hat{u}:\hat{H}\rightarrow\hat{G}$ is a surjection of the pro-$\mathcal{C}$ completions and $G$ has more than one end, $H$ also has more than one end.
\end{prop}

\begin{proof}
Now, by Stallings' Theorem (Theorem \ref{t: Stallings}), any group with more than one end splits as a graph of groups with a single edge and finite edge stabilizer. If a group $G$ is residually-$\mathcal{C}$ and has more than one end, then in this splitting, the vertex group or groups are also residually-$\mathcal{C}$ (as subgroups of $G$). Being residually-$\mathcal{C}$ means that the trivial subgroup, and hence any finite subgroup, is separable. Thus the edge group is separable in the vertex group(s) in this splitting and Lemma \ref{l: profinite splitting} gives an embedding of the tree $T$ associated with the splitting of $G$ into the pro-$\mathcal{C}$ tree $\hat{T}$ associated with the splitting of $\hat{G}$. Since the splitting of $G$ had finite edge stabilizers, this property carries through directly to the constructed splitting of $H$. Hence by Stallings' Theorem again $H$ has more than one end. 

\end{proof}

\begin{cor}\label{c: ends profinite}Having more than one end is a down-weak pro-$\mathcal{C}$ property. Equivalently, having one end is an up-weak profinite property. 
\end{cor}
\begin{proof}
This is just Proposition \ref{p: ends profinite} specialised to the case where $\hat{u}$ is an isomorphism.
\end{proof}

\begin{remark}
It is not in general the case that $H \backslash T$ recovers $\Gamma$ in the proof of Lemma \ref{l: profinite splitting}; indeed this occurs only when $H$ is of finite index in $G$, and since $\hat{u}$ is a surjection then $H$ can be of finite index only when $H=G$.
\end{remark}

Having two ends is the same as being virtually $\mathbb{Z}$. This is not hard to detect algebraically, and is in fact a pro-$\C$ property.

\begin{prop}
Let $\C$ be a variety of finite groups. Then having precisely 2 ends is a pro-$\C$ property.
\end{prop}
\begin{proof}



Let $H$ have 2 ends, and $G$ be pro-$\C$ equivalent to $H$. Since $H$ has two ends, it is virtually $\mathbb{Z}$, having a cyclic subgroup $H_0'$ of finite index; replacing $H_0'$ if necessary with its normal core we may assume this is normal in $H$. In particular $H$ is virtually abelian, and the closure $\bar{H_0'}$ of $H_0$ in the pro-$\C$ completion $\hat{H}$ is finite index $k$, normal, and abelian. Let $H_0 = \bar{H_0'}\cap H$, and $H_0 = \bar{H_0'} \cap G$ (since $G$ is residually $\C$ it injects into its pro-$\C$ completion $\hat{G}$, which coincides with $\hat{H}$, so this makes sense). So $G$ is also virtually abelian, and the index of $G_0$ in $G$ equals the index of $H_0$ in $H$ equals $k$.


As $G_0$ is a finitely generated abelian group, it is virtually $\mathbb{Z}^n$, for some $n$. If $n = 1$, $G$ is 2-ended and we are done. If $n > 1$, for a contradiction we will construct a finite quotient of $G$ which does not have a cyclic subgroup of index $\leq k$, which thus cannot be a quotient of $H$.

Say $G_0$ is $d$-generated and of rank $n$. Let $N$ be an integer much greater than $k$, such that the cyclic group $C_N \in \C$ (since $\C$ is subgroup closed, there is some $p$ such that $C_p \in C$, and since it is extension closed it follows that $C_{p^m} \in \C \forall m \in \mathbb{N}$, so such $N$ exists). Now $G_0^N$ is of index at least $N^n$ in $G_0$. It may not be the case that $G_0^N$ is normal in $G$, but it is certainly normal in $G_0$, and the quotient $G_0/G_0^N$ is a quotient of $C_N^d$ which surjects onto $C_N^n$. We let $L$ be the normal closure of $G_0^N$ in $G$. Since $G_0$ is of index $k$ in $G$, $L$ consists of the union of at most $k$ conjugates of $G_0^N$, so the index $[L:G_0^N]\leq k$. Since $G_0/(L \cap G_0)$ is a quotient of $G_0/G_0^N$, no cyclic subgroup has order greater than $N$. So $G/L$ has order at least $(N^n)/k$, and no cyclic subgroup of order greater than $Nk$. So long as $n \geq 2$ and $N > k^3$, it follows that $G/L$ has no cyclic subgroup of order $\leq k$, which is a contradiction.
\end{proof}

At least in the general case, however, we cannot hope for the stronger statement that the number of ends is a profinite property or even a weak profinite property, as the following example shows.

\begin{prop}
\label{p: parafree}
Having infinitely many ends is \emph{not} an up-weak pro-$p$ property for any prime $p$.
\end{prop}
\begin{proof}
This is an observation based on the theory of parafree groups. Parafree groups are residually nilpotent groups having the same nilpotent quotients as a free group. It is of interest that non-free parafree groups exist. We consider the group $G = \<a, b, c | a=[c, a][c, b]\>$, which was shown by Baumslag to be residually $p$ and (non-free) parafree \cite{baumslag1969groups}. He shows that $b$ and $c$ generate a free subgroup, and that their images generate any nilpotent quotient of $G$ (and hence in particular any finite $p$-quotient). Hence the injection $\<b,c\> \hookrightarrow G$ induces an isomorphism of pro-$p$ completions. $G$ is one ended; if not, it would split over a finite subgroup. But $G$, as a one-relator group, is torsion-free. Hence we would have $G$ as a free product. Grushko's theorem states that the rank of a free product is equal to the sum of the ranks of the free factors, hence in this case $G$ would be the free product of a cyclic group and a $2$-generator group. But $2$-generator subgroups of parafree groups are free \cite{baumslag1969groups}, so $G$ would be a free product of free groups, thus free, which is a contradiction.

\end{proof}

\newtheorem*{thm: profinite ends}{\emph{\textbf{Theorem \ref{t: profinite ends}}}}
\begin{thm: profinite ends}
\emph{If $\C$ is any class of finite groups closed under taking subgroups, quotients and extensions, then having more than one end is a down-weak pro-$\C$ property. However for any prime $p$ having more than one end is not an up-weak pro-$p$ property, even among finitely presented groups.}
\end{thm: profinite ends}
\begin{proof}
Combine the conclusions of Corollary \ref{c: ends profinite} and Proposition \ref{p: parafree}.
\end{proof}

\section{Producing further pathologies}\label{s: pathology machine}
There is a technique due to Rips which allows pathologies of an arbitrary finitely presented group to be translated to new pathologies of a group over which we have more control. The original Rips construction is as follows \cite{Rips}.

\begin{thm}\label{t: Rips construction}
For every $\lambda \geq 6$ there exists an algorithm that, given a finite presentation $\langle S | R \rangle$ for a finitely presented group $Q$ as input, outputs a short exact sequence $1\rightarrow N\rightarrow \Gamma\rightarrow Q \rightarrow 1$ such that:
\begin{enumerate}
\item $N$ is 2-generated; and
\item $\Gamma$ is a $C'(1/\lambda)$ small cancellation group and hence hyperbolic.
\end{enumerate}
\end{thm}

Since the original version several variations have been produced. These mostly vary the properties which $\Gamma$ is constrained to have, and sometimes demand more than two generators for $N$ (but still a finite number); they may also give extra control over $N$.

Wise exhibited a version in which $G$ is seen as a 1-acylindrical HNN extension of a free group, and used this characterisation to show that $G$ is residually finite \cite{Wise1}; Wilton and I subsequently showed that such groups are also conjugacy separable \cite{cotton2009conjugacy}.

It is a result due to Belegradek and Osin \cite{BelOsin} that there is a version of the Rips construction where the left group may be taken to be a quotient of an arbitrary non-elementary hyperbolic group:

\begin{theorem}\label{t: osin-rips}
There exists an algorithm which, given as input a finitely presented group $Q$ and a non-elementary hyperbolic group $H$, will output a short exact sequence \begin{equation}1 \rightarrow N \rightarrow \Gamma \rightarrow Q \rightarrow 1, \end{equation} where $\Gamma$ is finitely presented and hyperbolic, and $N$ is a quotient of $H$.
\end{theorem}

One might hope to combine this with the version discussed above in which the central output group is residually finite and conjugacy separable, in search of further pathological groups. However, the approach to proving residual finiteness is based upon expressing the group as the fundamental group of a graph of groups, with the normal subgroup $N$ necessarily not lying in a vertex group and thus having an action without fixed points on a tree. In Osin's construction, letting $H$ be a group with property (FA), which is inherited under quotients, it is clear that the two versions may not easily be reconciled.

If there is a genuine obstruction here, this would be very interesting, since careful application could give us explicit examples of hyperbolic groups which are not residually finite. On the other hand, if it could be made residually finite this would allow us to construct more examples demonstrating further properties not to be profinite. Of course if any of these properties \emph{were} known to be profinite, that would provide us with an obstruction to residual finiteness.

Haglund and Wise have described a further residually finite variant of the Rips construction \cite{Haglund-Wise}, which is known by recent work of Martino and Minasyan to be conjugacy separable \cite{Martino-Minasyan}. This suffers similar obstructions to combination with Osin's version: in this case $\Gamma$ is built as the fundamental group of a compact non-positively curved 2-complex, and acts freely on the universal cover $X$ of this complex, which is a CAT(0) space and hence contractible \cite{BridHaef}. By choosing $H$ with an appropriate fixed point property, we can see that the central output group from Osin's version of the Rips construction could never have such a structure. For instance for every $d > 0$ one can construct a non-elementary hyperbolic group which has no action without a global fixed point on any contractible space of dimension $\leq d$ \cite{arzhantseva2007infinite}. Taking $H$ to be such a group for $d=2$ suffices here, for if the two versions were reconciled then the free action of $\Gamma$ on the 2-complex $X$ would induce a free action of $N$ on $X$, and hence an action of $H$ on the $X$ without global fixed point, which is a contradiction.

\subsection{Direct profinite equivalence in short exact sequences of groups}
If a group $\Gamma$ surjects onto anything with no finite quotients, then all of the finite quotients of $\Gamma$ must be surjective images of the kernel. If moreover the kernel is finitely generated and the target group is superperfect (\emph{i.e.} having trivial homology in dimensions $1$ and $2$, with $\mathbb{Z}$ coefficients), this is a profinite equivalence:

\begin{prop}\label{p: profinite equivalence}
Let $1\rightarrow N\rightarrow \Gamma\rightarrow Q \rightarrow 1$ be a short exact sequence of groups such that $N$ is finitely generated, $\Gamma$ is finitely generated and residually finite, and $Q$ is superperfect and has no finite quotients. Then the injection of $N$ into $\Gamma$ induces a strong profinite equivalence.
\end{prop}
\begin{proof}
This is Lemma 5.2 in \cite{BridGrun}.
\end{proof}

We will use this source of strong profinite equivalences by inputting an appropriate group $Q$ to the version of the Rips construction in Theorem \ref{t: osin-rips}. Either the central group is residually finite, in which case we have a profinite equivalence and some properties not carried by it, or the central group is not residually finite, which would be interesting in its own right.
\newtheorem*{thm: unknown pathologies}{\emph{\textbf{Theorem \ref{t: unknown pathologies}}}}
\begin{thm: unknown pathologies}
\emph{Either not every Gromov hyperbolic group is residually finite, or neither of the following are up-weak profinite properties: property (FA); property (T).}
\end{thm: unknown pathologies}

Unfortunately I am thus far unable to say which of these alternatives is correct.

\begin{proof}
We let $P$ be an infinite, superperfect, finitely presented group with no finite quotients; such groups exist \cite{Bridson}. Now we take $Q$ = $P*P$; $Q$ is also an infinite, superperfect, finitely presented group with no finite quotients, and additionally it acts on a tree without global fixed point. We apply Theorem \ref{t: osin-rips} with this $Q$ as the input group and having chosen $H$ to be a hyperbolic group with property (T). This gives us an output short exact sequence $1 \rightarrow N \overset{u}{\rightarrow} G \rightarrow Q \rightarrow 1$ in which $G$ is hyperbolic and $N$ is a quotient of $H$, hence finitely generated and having (T) (since the latter is a fixed point property and so preserved under quotients). Note that (T) implies (FA) and that (FA) is inherited under quotients, so as $Q$ does not have (FA), neither can $G$.

We assume for a contradiction that $G$ is residually finite and that one of properties (FA) and (T) is up-weakly profinite. Since $G$ is residually finite then so is $N$ (as a subgroup), and so by Proposition \ref{p: profinite equivalence} we would have $\hat{u}$ an isomorphism of profinite completions. By the assumption, either (T) passes from $N$ to $G$, which thus has (FA), or (FA) passes from $N$ to $G$. In either case we conclude that $G$ has (FA), which is a contradiction; thus one of our assumptions was false.
\end{proof}

\begin{remark}Everything in the construction here is algorithmic (see \cite{Bridson-WiltonProfiniteTriviality}). Thus by choosing particular finite presentations for such groups $P$ and $H$ (which we can do!) we could produce a completely explicit finite presentation for the hyperbolic group $G$ which would be pathological in one of the above manners.
\end{remark}

Of course one can insert any property $X$ in place of (T) or (FA) if it satisfies a small number of conditions: $X$ must be inherited under quotients (this makes fixed point properties obvious candidates); there must exist a non-elementary hyperbolic group $H$ having $X$; there must exist a superperfect finitely presented group $Q$ having neither $X$ nor any nontrivial finite quotients. In order to prove that $G$ is a hyperbolic group which is not residually finite it would also suffice to show the (slightly weaker) statement that $N$ having property (T) would induce property (FA) in $G$.

Various results are in some sense close to showing that $G$ that is not residually finite. Properties (T) and ($\tau$) are not profinite \cite{Aka},\cite{Kassabov}, but ($\tau$) \emph{is} up-weakly profinite.\footnote{Property ($\tau$) says that among unitary representations which factor through finite quotients, the trivial representation is isolated, but adding elements to a generating set can only further isolate the trivial representation, so it is clear that it is up-weakly profinite.} We cannot use ($\tau$) in the above argument since we need $Q$ to be without finite quotients, so it automatically has ($\tau$); I cannot say whether (T) is up-weakly profinite. 

Having one end is also an up-weak profinite property, but it is not preserved under quotients. The proof that it is up-weak, Lemma \ref{l: profinite splitting}, almost suffices to show that property (FA) is also an up-weak profinite property: if the larger group acts on a tree $T$ then it induces an action on the profinite tree $\hat{T}$ associated with the profinite completion of the graph of groups; so long as this action is nontrivial then by continuity the subgroup will act on $T$ so as not to stabilise any vertex. As we showed there, it is sufficient for this that the edge groups be separable in the vertex groups. In fact the only case in which it fails is if the underlying graph of the graph of groups is simply connected and each edge group is dense in the vertex groups at which it is incident (all with respect to the profinite topology on $G$). In this case, the profinite graph of groups is simply labelled with $\hat{G}$ everywhere, and $\hat{T}$ is nothing but the underlying graph of the graph of groups, with the trivial action. Unfortunately this is precisely the situation we have in the above construction.

\section{Further questions}
We highlight two open questions. In the light of the results of Section \ref{s: ends profinite property} showing that having more than one end is a down-weak pro-$\C$ property for extension-closed $\C$ but not an up-weak pro-$p$ property, it is natural to ask if we can say anything stronger when $\C$ is the class of all finite groups.

\begin{question}
Is having more than one end an up-weak profinite property? Is the number of ends of a group a profinite property?
\end{question}

Having one end corresponds to having no action on a tree with finite edge stabilizers without fixed point. A natural generalisation is property (FA), simply having no action on a tree without fixed point. We ask whether the results of Section \ref{s: ends profinite property} can be extended to this case:

\begin{question}
Is having property (FA) an up-weak profinite property?
\end{question}

In view of the results of Section \ref{s: pathology machine} a positive answer to this question would be extremely interesting, as it would imply the existence of hyperbolic groups which are not residually finite.

\bibliographystyle{gillow}
\bibliography{refs}        

\end{document}